\documentclass[12pt, reqno]{amsart}

\usepackage{amssymb,latexsym,amsmath,amsfonts}
\usepackage{mathrsfs}
\usepackage{graphicx}
\usepackage[usenames]{color}
\usepackage{hyperref, cleveref}
\usepackage{comment}
\usepackage{enumitem}
\usepackage[normalem]{ulem}

\definecolor{DPurple}{rgb}{0.46,0.2,0.69}

\numberwithin{equation}{section}

\allowdisplaybreaks

\theoremstyle{definition}
\newtheorem{definition}{Definition}[section]

\theoremstyle{remark}
\newtheorem{remark}[definition]{Remark}

 \theoremstyle{plain}
\newtheorem{theorem}[definition]{Theorem}

\newtheorem{corollary}[definition]{Corollary}



\begin{document}

\title[Dynamics of complex harmonic mappings]{Dynamics of complex harmonic mappings}

\author{Gopal Datt}


\author{Ramanpreet Kaur}

\address{Department of Mathematics, Babasaheb Bhimrao Ambedkar University Lucknow, Uttar Pradesh, India-226025}
\email{ggopal.datt@gmail.com; gopal.du@bbau.ac.in}

\address{Department of Mathematical Sciences, Indian Institute of Science Education and Research Mohali, Knowledge City, Mohali, 140306, Punjab, India}
\email{ramanpreet@iisermohali.ac.in; preetmaan444@gmail.com }

\thanks{}


\begin{abstract}
This note initiates the study of the Fatou\,--\,Julia sets of a complex harmonic mapping. Along with some fundamental properties of the 
Fatou and the Julia sets, we observe some contrasting behaviour of these sets as those with in case of a holomorphic function. The 
existence of harmonic mapping with a wandering domain is also shown.
\end{abstract}
\maketitle

\vspace{-0.7cm}
\section{Introduction}\label{S:intro}

In this note, we develop the Fatou\,--\,Julia theory for a complex harmonic mapping under the direct-composition operation as 
introduced by Benítez-Babilonia and Felipe in \cite{Babilona2024}. The theory of complex harmonic mappings has 
proven valuable in various areas of modern science. 
For instance, in fluid mechanics, it is used to obtain explicit solutions to the incompressible two-dimensional Euler equations 
(see \cite{Aleman2012}).
Hence, it is important to study the behaviour of such a mapping from a point of view of dynamics. 
A  complex-valued function in a planar domain is said to be harmonic  if its real and imaginary parts are
harmonic functions, not necessarily harmonic conjugate. 
In what follows, a {\it harmonic mapping} always means a complex-valued harmonic function defined on the complex plane. 
A harmonic mapping 
$f$ in the complex plane can always be written as a sum $h+\bar{g},$ where $h$ and $g$ are 
holomorphic functions on $\mathbb{C}$.
In the literature, $h$ and $g$ are referred to as the analytic and co-analytic parts of $f$, respectively. 
We shall consider the family of 
harmonic mappings that do not preserve orientation, as the family of orientation-preserving harmonic mappings is always a 
normal 
family (see \cite{Duren2004}).
\smallskip

To study the dynamics of a harmonic mapping  
$f=h+\bar{g}$, let us first recall the direct-composition operation, denoted by $\circleddash$, which is defined as
\[f\circleddash f= (h+\bar{g})\circleddash (h+\bar{g})=h\circ h+\overline{g\circ g}.\]
 Therefore, for every $n\in\mathbb{N},$ we have  $f^{n,\circleddash}(z):=f\circ f^{n-1,\circleddash}(z)=h^n(z)+\overline{g^n(z)}$ for any 
 $z\in\mathbb{C}.$ It has also been coined by authors, in \cite{Babilona2024},
 that one can consider the cross-composition operation, denoted by $\circledcirc,$ defined as
 \[f\circledcirc f= h\circ g+\overline{g\circ h}.\]
However, the convergence of a sequence $\{f^{n,\circleddash}\}$ at a point $z_0$ implies the convergence of $\{f^{n,\circledcirc}\}$ at 
$z_0,$ i.e. the convergence of $\{f^{n,\circleddash}\}$ is a weaker notion (see \cite{Babilona2024}). Therefore, 
we shall consider only the direct-composition operation in this note. With this, we define the Fatou and the Julia sets of a harmonic mapping in the similar manner as they 
are defined  for a holomorphic function.

\begin{definition}
A family $\mathcal{F}$ of harmonic mappings is \emph{normal} 
at a point  $z_0\in \mathbb{C}$ if there exists a neighbourhood $U$ of  $z_0$ such that 
 for every sequence $\{f_n\}\subset\mathcal{F}$ there exists 
a subsequence $\{f_{n_k}\}$ which either converges uniformly on compact subsets of $U$ 
to a harmonic mapping or diverges compactly to $\infty$ in $U$. 
\end{definition}

A family $\mathcal{F}$ of harmonic mappings is said to be normal  on an open set 
$D$ if $\mathcal{F}$ is normal at each point of 
$D$.
It may also be noted that  the family of iterates $\{f^{n,\circleddash}\}$ is normal in $D$ 
 if and only if  it  is equicontinuous (see \cite[Lemma~1.1]{Wu}, and \cite[Page~10]{Duren2004}).
 \begin{definition}
The maximal open subset of $\mathbb{C}$ on which the family of iterates $\{f^{n,\circleddash}\}$ is normal is termed as the 
\emph{Fatou set} of $f$, and is denoted by $F(f).$
\end{definition}
We shall be using the definition of a point being in the Fatou set as the family of the iterates being either normal or 
 equicontinuous at that point.
The complement of the Fatou set is called as the \emph{Julia set}, and is denoted by $J(f).$ By definition, $J(f)$ is a closed set. 
A connected component of the Fatou set is called  a \emph{Fatou component}. A Fatou component may be
\begin{itemize}
\item a \emph{periodic component} if there exists $k_0\in\mathbb{N}$ such that $f^{k_0,\circleddash}(U)\subset U.$
\item a \emph{pre-periodic component} if there exists $m_0\in\mathbb{N}$ such that $f^{m_0,\circleddash}$ is periodic.
\item a \emph{wandering domain} if it is neither periodic nor pre-periodic.
\end{itemize}  
\smallskip

The wandering domains are further classified into three types. A wandering domain
\begin{itemize}
\item is \emph{escaping} if orbit of each of its points converges to infinity.
\item is \emph{oscillating} if each of its points has an unbounded orbit with a bounded subsequence.
\item is \emph{orbitally bounded} if  orbit of each of its points is bounded.
\end{itemize}
  \smallskip

In this article, we shall be mainly focussing on basic properties of the Fatou set, the Julia set, and the wandering domains.
We shall also
observe the contrast between the dynamics of harmonic and holomorphic functions. For the study of holomorphic dynamics, one can refer 
to \cite{Beardon1991,  Carleson1993, Fatou1919}, and for the study of harmonic mappings, see \cite{Duren2004}.

\section{Results on the dynamics of a complex harmonic mapping}
For a harmonic mapping $f=h+\bar{g}$, we shall establish two types of results in this section.
\begin{enumerate}
\item Some fundamental results on the Fatou and the Julia set of $f$. 

\item Existence of a harmonic mapping with wandering domains.
\end{enumerate} 

\subsection{Basic properties of the Fatou and the Julia sets}

\begin{theorem}\label{thm2.1}
 Let $f=h+\bar{g}$ be a harmonic mapping. Then $F(h)\cap F(g)\subseteq F(f)$, where $F(h)$ and $F(g)$ are 
 the Fatou sets of the holomorphic functions $h$ and $g$, respectively.
\end{theorem}

\begin{proof}
Let $z_0\in F(h)\cap F(g)$. For a given $\epsilon>0$ there exist $\delta_1, \delta_2>0$ such that we have the following:
\begin{align*}
|z-z_0|<\delta_1 &\text{ implies that } |h^n(z)-h^n(z_0)|<\frac{\epsilon}{2} \text{ for every }n\in\mathbb{N}, \text{ and }\\
|z-z_0|<\delta_2 &\text{ implies that } |g^n(z)-g^n(z_0)|<\frac{\epsilon}{2} \text{ for every }n\in\mathbb{N}.
\end{align*}
Take $\delta=\min\{\delta_1,\delta_2\}$. Then $|z-z_0|<\delta$ implies that
\begin{align*}
|f^{n,\circleddash}(z)-f^{n,\circleddash}(z_0)|&=|h^n(z)+\overline{g^n(z)}-h^n(z_0)-\overline{g^n(z_0)}|\\
&\leq |h^n(z)-h^n(z_0)|+|\overline{g^n(z)-g^n(z_0)}|\\
&<\frac{\epsilon}{2}+\frac{\epsilon}{2}=\epsilon, \text{ for every }n\in\mathbb{N}.
\end{align*}
Therefore, $z_0\in F(f).$
\end{proof}

\begin{remark}  It is worth noticing that
\begin{enumerate}
\item[(i)] the above theorem, in addition, gives us that $J(f)\subseteq J(h)\cup J(g).$
\item[(ii)] there exists a harmonic mapping for which we have a strict containment in the inclusion given in the 
\Cref{thm2.1}. To see this, take $$\displaystyle f(z)=z^2+\overline{\frac{z^2}{2}}.$$ 
Then  it is easy to see that
$$ \qquad  F(h)\cap F(g)=\{z\in\mathbb{C}: |z|<1\}\cup\{z\in\mathbb{C}: 1<|z|<{2}\}\cup\{z\in\mathbb{C}: |z|>{2}\}.$$
However, if we take any $z$ such that $|z|={2},$ then $h^n(z)\to\infty$ as $n\to\infty,$ and $g^n(z)$ is bounded, that is, 
there exists an $M>0$ such that $|g^n(z)|<M$ for every $n\in\mathbb{N}.$ 
This gives us that for any $z\in\mathbb{C}$ with $|z|={2},$ we have
\begin{align*}
|f^{n,\circleddash}(z)|&=|h^n(z)+g^n(z)|\\
&\geq |h^n(z)|-|g^n(z)|\\
&\geq |h^n(z)|- M\\
&\to\infty \text{ as }n\to\infty.
\end{align*}
This further implies that $z\in F(f).$
 
\end{enumerate} 
\end{remark}

\begin{theorem}\label{thm1}
$F(f^{p, \circleddash})=F(f)$ for every $p\in\mathbb{N}.$ 
\end{theorem}

\begin{proof}
We shall prove this for $p=2.$  It follows, from the definition, that $F(f)\subseteq F(f^{2, \circleddash}).$ 
Now let $z_0\in F(f^{2, \circleddash}).$ It means that $\{f^{2n, \circleddash}\}$ is equicontinuous at $z_0.$ 
If we show that the family $\{f^{2n+1,\circleddash}: n\geq 0\}$ is equicontinuous at $z_0$, then we are done as the union of two 
equicontinuous families is equicontinuous. To show this, first we observe the following.
\begin{enumerate}
\item By the equicontinuity of $\{f^{2n,\circleddash}\}$ at $z_0$, 
if $\epsilon>0$ is given, then corresponding to ${\epsilon}/{3}$, there exists a $\delta_1>0$ such that for every $n\in\mathbb{N}$,
\[|f^{2n,\circleddash}(z)-f^{2n,\circleddash}(z_0)|<\frac{\epsilon}{3}, \text{ whenever } |z-z_0|<\delta_1.\]
\item Continuity of $h$ and $g$ at $z_0$ gives the existence of $\delta_2>0$ such that
\[|z-z_0|<\delta_2 \text{ implies that } |h(z)-h(z_0)|<\delta_1 \text{ and }|g(z)-g(z_0)|<\delta_1.\]
\item Continuity of $g+\overline{h}$ at $z_0$  gives the existence of $\delta_3>0$ such that
\[|z-z_0|<\delta_3 \text{ implies that } |(g+\overline{h})(z)-(g+\overline{h})(z_0)|<\delta_1. \]
\end{enumerate}
Now consider, 
\begin{align*}
|f^{2n+1,\circleddash}(z)-f^{2n+1,\circleddash}(z_0)|&=|h^{2n+1}(z)+\overline{g^{2n+1}(z)}-h^{2n+1}(z_0)-\overline{g^{2n+1}(z_0)}|\\
&=|h^{2n+1}(z)
+\overline{g^{2n}(h(z))}-\overline{g^{2n}(h(z))}\\
&\quad+\overline{g^{2n}(h(z_0))}-\overline{g^{2n}(h(z_0))}+h^{2n}(g(z))\\
&\quad-h^{2n}(g(z))+h^{2n}(g(z_0))
-h^{2n}(g(z_0))\\
&\quad+\overline{g^{2n+1}(z)}-h^{2n+1}(z_0)-\overline{g^{2n+1}(z_0)}|\\
&\leq |(h^{2n}+\overline{g^{2n}})(h(z))-(h^{2n}+\overline{g^{2n}})(h(z_0))|	\\
&\quad+ |(h^{2n}+\overline{g^{2n}})(g(z))-(h^{2n}+\overline{g^{2n}})(g(z_0))|\\
&\quad+|(h^{2n}+\overline{g^{2n}})\circleddash(g+\overline{h})(z))\\
&\quad-(h^{2n}+\overline{g^{2n}})\circleddash(g+\overline{h})(z_0))|\\
&<\frac{\epsilon}{3}+\frac{\epsilon}{3}+\frac{\epsilon}{3}=\epsilon.	
\end{align*}
This completes the proof that $F(f^{2, \circleddash})=F(f)$. By similar arguments, one can show that $F(f^{p, \circleddash})=F(f)$ for every 
$p\in\mathbb{N}.$
\end{proof}

\Cref{thm1} raises a natural question: {\it
If $f_1$ and $f_2$ are two permutable harmonic mappings, 
then do their Fatou sets coincide?} Before delving into this question further, 
first we shall explain what do we mean by $f_1$ and $f_2$ being permutable.

\begin{definition}
Let $f_1=h_1+\bar{g_1}$ and $f_2=h_2+\bar{g_2}$ be two harmonic mappings. We say that $f_1$ and $f_2$ are 
\emph{permutable} if both of their analytic and co-analytic parts commute, that is, $h_1\circ h_2=h_2\circ h_1$ and $g_1\circ 
g_2=g_2\circ g_1.$
\end{definition}

\begin{theorem}\label{thm2}
If $f_1, f_2$ are two permutable harmonic mappings such that $f_1(z)=h_1(z)+\overline{g_1(z)}$, and $f_2(z)=af_1(z)+b,$ where $a, b\in\mathbb{C}$ with $|a|=1$, then $F(f_1)=F(f_2).$
\end{theorem}
\begin{proof}
Let $z_0\in F(f_1)$. By definition, for a given $\epsilon>0$ there exist a $\delta>0$ such that for every $n\in\mathbb{N}$, we have $|f_1^{n,\circleddash}(z)-f_1^{n,\circleddash}(z_0)|<\epsilon$ whenever $|z-z_0|<\delta$.  Now observe that 
\[f_2^{2,\circleddash}(z)=ah_1(ah_1(z)+b)+b+\overline{\bar{a}g_1(\bar{a}z)}.\] 
On using the fact that $f_1\circleddash f_2=f_2\circleddash f_1$, we get that $h_1(ah_1(z)+b)=ah_1^2(z)+b$ and $g_1(\bar{a}z)=\bar{a}g_1(z).$ This further gives us that 
\[f_2^{2,\circleddash}(z)= a(ah_1^2(z)+b)+b+\overline{\bar{a}^2g_1^2(z)}.\] 
Inductively for any $n\in\mathbb{N}$, we get 
\[f_2^{n,\circleddash}(z)=a^nh_1^n(z)+b(a^{n-1}+a^{n-2}+\dots+1)+\overline{\bar{a}^ng_1^n(z)}.\]
Consider 
\begin{align*}
|f_2^{n,\circleddash}(z)-f_2^{n,\circleddash}(z_0)|&=|a^nh_1^n(z)+b(a^{n-1}+a^{n-2}+\cdots+1)+\overline{\bar{a}^ng_1^n(z)}\\
&\quad -a^nh_1^n(z_0)-b(a^{n-1}+a^{n-2}+\cdots+1)-\overline{\bar{a}^ng_1^n(z_0)}|\\
&=|a|^n|f_1^{n,\circleddash}(z)-f_1^{n,\circleddash}(z_0)|\\
&<\epsilon, 
\end{align*}
whenever $|z-z_0|<\delta,$ for every $n\in\mathbb{N}.$
This shows that $z_0\in F(f_2)$. On similar lines, one can show that $F(f_2)\subseteq F(f_1).$ Hence,  $F(f_2)=F(f_1)$.
\end{proof}

\begin{remark}\label{rmk2.6}
From \Cref{thm2}, one can deduce that if $f_1(z)=h_1(z)+\overline{g_1(z)}$ and $f_2(z)=h_1(z)+p+\overline{g_1(z)},$ where $p$ is a 
period of the function $h_1$, then $F(f_1)=F(f_2)$.
\end{remark}

\subsubsection*{\bf \emph{Contrasting behaviour}}
In holomorphic dynamcis, both the Fatou set and the Julia set are completely invariant. Recall that a set $E\subseteq \mathbb{C}$ is said 
to be {\it completely invariant} under $f$ if and only if $f(E)\subseteq E$ and $f^{-1}(E)\subseteq E.$ This is not true for a harmonic 
mapping. For example, take $ f(z)=z^2+\overline{\frac{z^2}{2}}.$ Any point $z\in\mathbb{C}$ such that $|z|=1$ belongs to the Julia set of 
$f$, and image of any such point under $f$ need not belong to the Julia set of $f$.  To see this, fix a point $z_0$ such that $|z_0|=1$. As $z_0\in 
J(z^2),$  there exists $\epsilon_1>0$ such that for all $\delta>0$ there exist $z_\delta\in \mathbb{C}$ 
and $n_\delta\in \mathbb{N}$
such that 

\begin{equation}\label{2.1}
|z_\delta^{2^{n_\delta}}-z_0^{2^{n_\delta}}|\geq\epsilon_1 \text{ whenever } |z_\delta-z_0|<\delta.
\end{equation}

Also, as $z_0\in F\left(\frac{z^2}{2}\right),$  there exists a $\delta_1>0$ such that for every $n\in\mathbb{N},$
\[\left|\frac{z^{2^{n}}}{2^{n+1}}-\frac{z_0^{2^{n}}}{2^{n+1}}\right|<\frac{\epsilon_1}{2}\text{ whenever }|z-z_0|<\delta_1.\]
If $z_0\in F(f)$, then for every $\epsilon>0$, there exists a $\delta_2>0$ such that for every $n\in\mathbb{N}$,
\begin{equation}\label{2.2}
|f^{n,\circleddash}(z)-f^{n,\circleddash}(z_0)|<\frac{\epsilon_1}{2} \text{ whenever }|z-z_0|<\delta_2. 
\end{equation}
The above \Cref{2.2} is also true, in case we take $\delta=\min\{\delta_1,\delta_2\}$.
This gives us
\begin{align*}
|z^{2^{n}}-z_0^{2^{n}}|-\left|\frac{z^{2^{n}}}{2^{n+1}}-\frac{z_0^{2^{n}}}{2^{n+1}}\right|&<|f^{n,\circleddash}(z)-f^{n,\circleddash}(z_0)|<\frac{\epsilon_1}{2}.
\end{align*}
This implies that for every $n\in \mathbb{N}$
\[|z^{2^{n}}-z_0^{2^{n}}|<\frac{\epsilon_1}{2} \text{ whenever }|z-z_0|<\delta,\]
which is a contradiction to \Cref{2.1}.
This establishes that any point on the unit circle is in the Julia set of $f$. However,
for $z=1$, $f(z)=3/2$ which belongs to  the Fatou set of $f$, by \Cref{thm2.1}.
\smallskip

 Next, recall that for a rational function of degree at least $2$, the Julia set is always non-empty. In case of a harmonic mapping, it may not be true. For 
 example, take $f(z)=z^2+\overline{2z}$. It is easy to see that for every $z\in\mathbb{C}$, the sequence of iterates $f^{n,\circleddash}(z)
 \to\infty$ as $n\to\infty$. Therefore, $F(f)={\mathbb{C}}$, and hence $J(f)=\emptyset.$ However,
if we take a harmonic mapping $f=h+\bar{g}$ such that there exists a  $z\in J(h)\cap \mathsf{int}(K(g))$ or 
$z\in J(g)\cap \mathsf{int}(K(h))$, then $J(f)
\not=\emptyset.$ Here, $K(g)$ denotes the set of all those points whose orbit under $g$ remains bounded. Therefore, when we do 
dynamics, we shall always assume that a harmonic mapping under consideration has a non-empty Julia set.

\subsection{Existence Results}
We first define some terminologies before demonstrating results in the direction of wandering domains.
A harmonic mapping $f=h+\bar{g}$ is said to be a \emph{polynomial (rational)} harmonic mapping if both $h$ and $g$ are polynomials (rationals) of degrees atleast $2$, and a \emph{transcendental} harmonic mapping if at least one of the $h$ and $g$ is a transcendental function.

\begin{theorem}\label{thm3}
If $f=h+\bar{g}$ is a harmonic mapping such that both $h$ and $g$ are polynomials of degrees $d_1$ and $d_2$ respectively, then $f$ will neither have an escaping wandering domain nor an oscillating wandering domain.
\end{theorem}
\begin{proof}
Since both $h$ and $g$ are polynomials, there exist $r_1, r_2>0$ such that 
\begin{align*}
h^n(z)\to\infty &\text{ whenever }|z|>r_1, \text{ and }\\
g^n(z)\to\infty &\text{ whenever }|z|>r_2.
\end{align*}
This implies that $f^{n,\circleddash}(z)\to\infty$ whenever $|z|>r=\max\{r_1,r_2\}.$ Therefore, $f$ can neither have escaping wandering domains nor oscillating wandering domains.
\end{proof}

\begin{remark} From the above theorem, we can also conclude that
\begin{itemize}
\item[(i)] the Fatou set of a polynomial harmonic mapping always has an unbounded component. This further implies that its Julia set is a 
 compact subset of $\mathbb{C}.$ \label{rmk2.8(1)}
 
\item[(ii)]  if we take a transcendental complex harmonic mapping, then it may have escaping wandering domains.
\end{itemize}
\end{remark}

\begin{theorem}
There exists a harmonic mapping with an escaping wandering domain.
\end{theorem}

\begin{proof}
Let $g_1$ be a holomorphic function on the complex plane such that $g_1(2m\pi\iota)=0$ for every $m\in\mathbb{Z}$ and $0$ is an attracting fixed point for $g_1,$ that is $g_1(0)=0$ and $0<|g_1'(0)|<1.$ Take 
$f_1(z)=h_1(z)+\overline{g_1(z)},$
where $h_1(z)=z-1+e^{-z}$.
For every $m\in\mathbb{Z},$ we have $2m\pi\iota$ is an attracting fixed point of $h_1$. Also, for a holomorphic function, attracting periodic points lie in the Fatou set, and the Fatou set is completely invariant 
(see \cite[page 54, Theorem~3.2.4]{Beardon1991}). Fix  $m\in\mathbb{N}$. Then for a given $\epsilon>0,$ there exists $\delta_1,\delta_2>0$ such that we have 
\begin{align*}
|z-2m\pi\iota|<\delta_1& \text{ implies that }|h_1^n(z)-2m\pi\iota|<\frac{\epsilon}{2}, \text{ and }\\
|z-2m\pi\iota|<\delta_2& \text{ implies that }|g_1^n(z)|<\frac{\epsilon}{2},
\end{align*}
for every $n\in\mathbb{N}.$
Take $\delta=\min\{\delta_1, \delta_2\}$. Then, for every $n\in\mathbb{N},$ we have
\begin{align*}
|f_1^{n,\circleddash}(z)-f_1^{n,\circleddash}(2m\pi\iota)|&=|h_1^n(z)+\overline{g_1^n(z)}-2m\pi\iota|\\
&<\frac{\epsilon}{2}+\frac{\epsilon}{2}=\epsilon, 
\end{align*}
whenever $|z-2m\pi\iota|<\delta.$
Now, if we take $f_2(z)=f_1(z)+2\pi\iota$, then using \Cref{rmk2.6}, we get that $F(f_1)=F(f_2).$ As $f_2(2m\pi\iota)=2(m+1)\pi\iota$ 
for every $m\in\mathbb{N},$ we have that the Fatou component containing $0$ is an escaping wandering domain.
\end{proof}
Remark $2.8$, reveals that a polynomial harmonic mapping can never have all of its Fatou components bounded. This behaviour 
is analogous with that of polynomials. In \cite{Baker1980}, Baker posed a question: {\it If $h$ is a transcendental entire function of order at 
most ${1}/{2}$, minimal type, then whether $F(h)$ has any unbounded Fatou component.} We ask the same question for a 
transcendental harmonic mapping. It is to be mentioned that for a transcendental entire function with order strictly less than ${1}/{2}$ 
and minimal type, an affirmative answer has been provided in \cite{Raman2023}. In case of a transcendental harmonic mapping, we try to 
employ the similar techniques. We first recall a few definitions  before proceeding in this direction.
\begin{definition}[\cite{Langley2007}]
Let $h$ be a holomorphic function. The order of the function $h$, denoted by $\rho(h)$ is defined as
\[\rho(h)=\limsup\limits_{r\to\infty}\frac{\log\log M(r,h)}{\log r},\]
where $M(r,h):=\max\{|h(z)|: |z|\leq r\}.$
\end{definition}
\noindent We say a holomorphic function $h$ is of {\it minimal type} if 
\[\limsup\limits_{r\to\infty}\frac{\log M(r,h)}{\log r^{\rho(h)}}=0.\]

It is known that if  $h$ and $g$ are any holomorphic functions with finite order, then $\rho(h +g)\leq \max\{\rho(h), \rho(g)\},$ and 
if  $\rho(h)<\rho(g),$ then $\rho(h +g)=\rho(g).$
Since $|g(z)|=|\overline{g(z)}|$, the order $\rho(f)\leq \max\{\rho(h), \rho(g)\}$, where $f=h+\bar{g}.$  Further, if $g$ is a polynomial, then $
\rho(f)=\rho(h).$ For more details, one can refer to \cite{Langley2007}. The following result gives us a sufficient condition for all the Fatou 
components to be bounded.
 
\begin{theorem}\label{thm2.11}
Let $f=h+\bar{g}$ be a transcendental harmonic mapping such that the following conditions hold
\begin{enumerate}
\item both $F(f)$ and $J(f)$ are forward invariant.
\item $J(f)$ contains at least two points.
\item $|g^n(z)|\leq |h^n(z)|$ for every $n\in\mathbb{N}.$
\item there exist sequences $r_k, \sigma_k\to\infty$ and $s(k)>1$ such that 
\begin{enumerate}
\item $M(r_k, h)= r_{k+1},$
\item $r_k\leq \sigma_k\leq r_{k}^{s(k)}$,
\item $m(\sigma_k, h)>r_{k+1}^{s(k+1)}$ for every sufficiently large $k$. Here, $m(\sigma_k, h):=\min\{|h(z)|:|z|=\sigma_k\}.$
\end{enumerate}
\end{enumerate} 
  
Then the Fatou set of $f$ does not have any unbounded  component.
\end{theorem}
Prior to the proof of \Cref{thm2.11}, we first make a few comments about its hypotheses as follows.
\begin{itemize}
\item[(i)] In condition $(1)$, the forward invariance of these sets is taken so that we can employ techniques from holomorphic dynamics. 
In addition, it seems a bit difficult to even formulate such a question for a general harmonic mapping.
\item[(ii)] The condition $(3)$ gives us that $h$ is a transcendental entire function.
\item[(iii)] The other conditions are motivated from \cite{Baker1980}.

\end{itemize}  
\noindent\textit{Proof of the \Cref{thm2.11}:}
Assume that $F(f)$ has an unbounded component, say $U$. Also, without loss of generality, assume that $0,1 \in J(f)$, and $f(0)=1.$ 
By $(1)$, each $f^{n,\circleddash}$ omit values $0$ and $1$ in $U$. Now, we follow the following steps to arrive at a 
contradiction.
\smallskip

\noindent {\it Step 1:} From $(c)$, there exists $k_0\in\mathbb{N}$ such that $m(\sigma_k, h)>r_{k+1}^{s(k+1)}$ for every $k\geq k_0.$
\smallskip

\noindent 
{\it Step 2:} Since $U$ is path connected and an unbounded component, there exists  $k_1\in\mathbb{N}$ such that $U$ intersects the 
following three circles for every $k\geq k_1.$
\begin{align*}
T_k&=\{z\in\mathbb{C}: |z|=p_k\},\\
T_k^1&=\{z\in\mathbb{C}: |z|=r_k^{s(k)}\},\\
T_k^2&=\{z\in\mathbb{C}:|z|=\sigma_k\}.
\end{align*}

\noindent  In $T_k$, the sequence $\{p_k\}$ is taken such that $M(p_k,h)=p_{k+1}$ and $p_k^2=r_k$ for every $k\in\mathbb{N}.$%
\smallskip

\noindent 
{\it Step 3:} For a fixed $k\geq \max\{k_0, k_1\},$ we have the existence of a path $\Gamma$ in $U$ joining the points $z_k\in T_k$ and 
$z_{k+1}^1\in T_{k+1}^1.$ This implies that $\Gamma$ will intersect $T_{k+1}^2$ at some point, say $z_{k+1}^2.$
\smallskip

\noindent {\it Step 4:} By $(1)$, the set $f(U)$ is also a Fatou component of $f$. 
It will be an unbounded component as well. For, if $f(U)$ 
is a bounded component, then there exists $M>0$ such that $|f(z)|<M$ for every $z\in U.$ Now, for every $k\geq k_1,$ there exists  
$w_k\in T_k^2 \cap U$ such that $r_k^{s(k)}<|h(w_k)|<|f(w_k)|<M$. Since $r_k\to\infty$ as $k\to\infty$, this gives us a contradiction. 
Therefore, $f(U)$ has to be an unbounded Fatou component.
\smallskip

\noindent {\it Step 5:} The set $f(U)$ contains  the path $f(\Gamma).$ Further, 
$|f(z_k)|\leq 2M(p_k,h)=2p_{k+1}\leq p_{k+1}^2=r_{k+1},$ and from 
$(c),$ we have $|f(z_{k+1}^2)|>r_{k+2}^{s(k+2)}.$ This gives that $f(\Gamma)$ contains an arc joining the points $z_{k+1}$  and 
a point $z_{k+2}^1\in T_{k+2}^1$
such that $|z_{k+1}|\leq 2p_{k+1}$. By induction, we get $f^{n, \circleddash}(U)$ is an unbounded component, and 
contains an arc of $f^{n, \circleddash}(\Gamma)$ joining the points $z_{k+n}$ such that 
$|z_{k+n}|\leq (n+1)p_{k+n}$ and $z_{k+n+1}^1\in 
T_{k+n+1}^1$ for every $n\in\mathbb{N}.$ 
\smallskip

\noindent {\it Step 6:} On the path $\Gamma,$ we have $|f^{n, \circleddash}(z)|\geq (n+1)p_{k+n}$ which tends to infinity as 
$n\to\infty$ . 
Therefore, 
$f^{n,\circleddash}\to\infty$ as $n$ tends to infinity locally uniformly in $U$. This further implies the existence of an $N_0\in\mathbb{N}$ 
such that for every $n> N_0$, we have $|f^{n,\circleddash}(z)|>2$ for every $z\in\Gamma.$ By hypothesis, $|h^n(z)|\geq |g^{n}(z)|$ for 
every $n\in\mathbb{N}$. This gives that for every $n>N_0, |h^n(z)|>1$ for all $z\in\Gamma.$

\smallskip

\noindent 
{\it Step 7:} Since $\log |h^n|$ is a positive harmonic mapping on $\Gamma,$ for every $n>N_0$, there exists a $\alpha>0$ 
such that for 
every $u, v\in\Gamma,$ we have
\[|h^n(u)|\leq |h^n(v)|^\alpha.\]
Now for each $n>N_0$, choose $u_n, u_n^1\in\Gamma$ such that $f^{n,\circleddash}(u_n)=z_{k+n}$ and 
$f^{n,\circleddash}(u_n^1)=z_{k+n+1}^1.$ Therefore, for every $n>N_0,$
\[ |h^n(u_n^1)|\leq|f^{n,\circleddash}(u_n^1)|\leq 2 |h^n(u_n^1)|\leq 2 |h^n(u_n)|^\alpha\leq 2 |f^{n,\circleddash}(u_n)|^\alpha.\]
Further, 
\begin{align*}
2(n+1)r_{k+n}^{\alpha/2}>2(n+1)p_{k+n}^\alpha &\geq r_{k+n+1}^{s(k+n+1)}> r_{k+n+1}=M(r_{k+n},h),
\end{align*}
for every $n>N_0.$
This gives 
\[\frac{\log M(r_{k+n}, h)}{\log r_{k+n}}<\infty,\]
for every $n>N_0,$ which is a contradiction as $h$ is a transcendental entire function. 
Therefore, $F(f)$ does not have any unbounded component.\qed

\begin{remark}
Using the condition $(3)$, one can deduce that $\rho(g)\leq \rho(h)$. This, in particular, gives us that $\rho(f)\leq \rho(h).$
\end{remark}

\begin{corollary}
If $f=h+\bar{g}$ is a transcendental harmonic mapping such that the order of $h$ is strictly less than ${1}/{2},$ minimal type and the conditions $(1)-(3)$ hold, then $F(f)$ does not have any unbounded component.
\end{corollary}

\begin{proof}
From \cite{Raman2023}, we get that the condition $(4)$ is satisfied for $h$. Therefore, the result follows by the \Cref{thm2.11}.
\end{proof}

\section*{Acknowledgements}\vspace{-1mm}
The research of the second author was supported by the Institute Post-doctoral Fellowship
program at Indian Institute of Science Education and Research (IISER) Mohali. 
The authors are also thankful to Prof. Sanjay Kumar for his suggestions.

\section*{Declarations}

\textbf{Conflict of interest:} The authors declare that they do not have any conflict of interest.
\smallskip

\textbf{Data Availability:} Data availability is not relevant.

\end{document}